\documentclass[reqno,12pt]{amsart}
\usepackage{amsmath,amsthm,amssymb}
\usepackage{xcolor}
\usepackage[colorlinks]{hyperref}

\newtheorem{theorem}{Theorem}[section]
\newtheorem{lemma}[theorem]{Lemma}
\newtheorem{corollary}[theorem]{Corollary}
\newtheorem{remark}[theorem]{Remark}

\theoremstyle{definition}
\newtheorem{definition}[theorem]{Definition}

\numberwithin{equation}{section}

\begin{document}

\title[Functional analogs]
{Functional analogs of the Shephard, Busemann--Petty, and Milman problems}

\author{Vadim Gorev}


\author{Egor Kosov}


\maketitle

\begin{abstract}
The paper studies possible functional analogs of classical
problems from convex geometry. In particular, we provide
some bounds in the functional Shephard, Busemann--Petty, and Milman
problems generalizing known bounds in this problems for convex sets.
\end{abstract}



\section{Introduction}
\label{sect-Int}

Classical Shephard's problem
(see \cite{Sh}) asks
whether one has
$$
\lambda_n(K)\le \lambda_n(V)
$$
for a pair of convex bodies $K, V\subset \mathbb{R}^n$
such that
\begin{equation}\label{eq-Sh-cond}
 \lambda_{n-1}(P_H(K))\le \lambda_{n-1}(P_H(V))
\end{equation}
for every $n-1$ dimensional hyperplane $H$.
Here $P_H$ denotes the orthogonal
projection operator
on the hyperplane $H$ and
$\lambda_n$ denotes the standard $n$-dimensional
Lebesgue measure.
Similarly, the classical Busemann--Petty problem (see \cite{BP})
asks whether one has
$$
\lambda_n(K)\le \lambda_n(V)
$$
for a pair of centrally symmetric convex bodies $K, V\subset \mathbb{R}^n$
such that
\begin{equation}\label{eq-BP-cond}
\lambda_{n-1}(K\cap H)\le \lambda_{n-1}(V\cap H)
\end{equation}
for every $n-1$ dimensional hyperplane $H$.
It is now known that both problems have negative answers for sufficiently big $n$
(see the monograph \cite{Kold} for a detailed discussion).

For every dimension $n$,
the isomorphic versions of the stated problems seek to find
the best possible constant $C_n$ such that
$$
\lambda_n(K)\le C_n\lambda_n(V)
$$
for every pair of convex bodies $K, V\subset \mathbb{R}^n$
under the assumption \eqref{eq-Sh-cond} for the
isomorphic Shephard problem and under the ussumption \eqref{eq-BP-cond}
for the
isomorphic Busemann--Petty problem.
It is known that the isomorphic Busemann--Petty problem is closely
related to Bourgain's slicing conjecture (see Section 3.5.2 in \cite{BGVV}).
For the constant $C_n$ in the isomorphic Shephard problem,
K. Ball (see \cite[Section 2]{B91}) showed that $c_1\sqrt{n}\le C_n\le c_2\sqrt{n}$
for some absolute constants $c_2>c_1>0$.

The variant of these two problems, proposed by V. Milman, has been studied in \cite{GK}.
The Milman variant asks whether one has
$$
\lambda_n(K)\le \lambda_n(V)
$$
for a pair of convex bodies $K, V\subset \mathbb{R}^n$
such that
$$
 \lambda_{n-1}(P_H(K))\le \lambda_{n-1}(V\cap H)
$$
for every $n-1$ dimensional hyperplane $H$.
The paper \cite{GK} gives an affirmative answer to this question.

In this paper we study functional analogs of the geometric problems described above.
Functional analogs of geometric notions and inequalities have been extensively studied in recent years
(e.g. see
\cite{A-GMJV16}, \cite{A-GMJR}, \cite{A-G}, \cite{A-GA-AMJV}, \cite{A-GBM}, \cite{A-AKM},
\cite{BCF}, \cite{C06}, \cite{CF}, \cite{KM}
and citations therein).
Obviously, even the classical Shephard's and
Busemann--Petty problems are meaningless without any convexity assumption.
Thus, we restrict ourselves mainly to the class of logarithmically concave functions
which is one of the natural extensions of convex bodies to functional setting.
Our initial interest has been concerned with the functional Shephard problem
since the natural counterpart of a projection of a body to a subspace
$H=\langle\theta\rangle^\bot$, $|\theta|=1$,
is $L^1$-norm of a derivative of a function $\|\partial_\theta f\|_{L^1(\lambda_n)}$.
Thus, the functional Shephard
problem asks what we can say about the functions $f_1, f_2$ such that
$$
\|\partial_\theta f_1\|_{L^1(\lambda_n)}\le \|\partial_\theta f_2\|_{L^1(\lambda_n)} \quad \forall \theta \in \mathbb{R}^n, |\theta| = 1.
$$
It appears that the connection between $f_1$ and $f_2$
is not through the inequality between $L^1$-norms
of these functions but it rather involves a combination
of different $L^p$-norms of $f_1$ and~$f_2$.

The main result concerning the functional Shephard problem
is presented in Section~\ref{sect-Sh-1} (see Theorem~\ref{T1}).
In Section~\ref{sect-Sh-2} we study the functional Shephard-type problem
when the assumptions are imposed on projections on subspaces of a  codimension
bigger than one.
The main result here is stated in Theorem \ref{T2}.
The main tools in Sections \ref{sect-Sh-1} and \ref{sect-Sh-2} are the functional analog
of the John ellipsoid that was constructed in \cite{A-GMJR}
and the functional version of Aleksandrov’s inequality from \cite{BCF}.
The proofs here are similar to K. Ball argument from~\cite{B91}.
In Section \ref{sect-BP} we consider the functional analog of the
Busemann--Petty problem (see Theorems \ref{T4} and \ref{T5})
and
Section~\ref{sect-M}
is devoted to the functional Milman problem (see Theorems \ref{T3} and \ref{T6}).
The arguments in the last two sections follow some ideas from \cite{GK}.
One of the main tools in these two sections is
functional Grinberg's inequality which is proved in Theorem \ref{funcgrinberg}.

The following section presents the notation used in the present paper,
and contains several known results that will be applied throughout the paper.

\section{Definitions, notation and preliminary results}
\label{sect-N}

Let $\langle\cdot, \cdot\rangle$ be the standard inner product on
$\mathbb{R}^n$ and let $|\cdot|$ be the standard Euclidean norm in $\mathbb{R}^n$,
i.e. $|x|:=\sqrt{\langle x, x\rangle}$. Let $\lambda_n$ be the standard Lebesgue
measure on $\mathbb{R}^n$. The volume of the unit Euclidean ball
$B_2^n:=\{x\colon |x|\le 1\}$ in $\mathbb{R}^n$
is denoted by $\omega_n$,
i.e. $\omega_n:=\lambda_n(x\colon |x|\le 1)$.
The space $L^p(\mathbb{R}^n)$ consists of all
functions integrable to the power $p\in[1, +\infty)$.
The space $L^p(\mathbb{R}^n)$ is
equipped with the standard norm
$$
\|f\|_{L^p(\mathbb{R}^n)}=\|f\|_p:=\Bigl(\int_{\mathbb{R}^n}|f(x)|^p\, dx\Bigr)^{1/p}.
$$
Let $\|f\|_\infty:=\sup\limits_{x\in\mathbb{R}^n}|f(x)|$
for a bounded function $f$ on $\mathbb{R}^n$.

Recall that the function $f\colon \mathbb{R}^n\to [0, +\infty)$
is called logarithmically concave (log-concave) if
$$
f(tx+(1-t)y)\ge [f(x)]^t\cdot[f(y)]^{1-t}\quad \forall x, y\in \mathbb{R}^n, \forall t\in(0, 1).
$$
In particular, the sets $A_t(f):=\{x\in \mathbb{R}^n\colon f(x)\ge t\}$
are convex for every logarithmically concave function $f$.
The main object, studied in this paper,
is the class of all logarithmically concave functions on $\mathbb{R}^n$.
Namely, we consider the class
$$
Q_0^n:=\{f\colon \mathbb{R}^n\to [0, +\infty)\colon f \hbox{--- log-concave, upper semicontinuous},
f\in L^1(\mathbb{R}^n), \|f\|_1 > 0\}.
$$
We recall that a function $f\colon\mathbb{R}^n\to \mathbb{R}$
is called upper semicontinuous if the sets
$$
A_t(f):=\{x\in \mathbb{R}^n\colon f(x)\ge t\}
$$ are closed for each $t\in\mathbb{R}$.
Equivalently, one has $\limsup\limits_{n\to\infty}f(x_n)\le f(x)$
for every sequence $x_n\to x$.
We point out, that for a function $f \in Q_0^n$ the stated
inequality is actually an equality.
We also recall (see \cite[Lemma~2.2.1]{BGVV}) that, for every function
$f\in Q_0^n$, there are numbers $A, B>0$ such that
\begin{equation}\label{eq-exp}
f(x)\le Ae^{-B|x|}\quad \forall x\in \mathbb{R}^n.
\end{equation}
In particular, any function $f\in Q_0^n$ is bounded and $\lim\limits_{|x|\to \infty}f(x) = 0$.
Due to the upper semicontinuity assumption and due to the above mentioned bound,
every function from $Q_0^n$ attains its maximum.

Let $G_{n, k}$ denote the Grassmannian of all $k$-dimensional subspaces in $\mathbb{R}^n$
and let $\nu_{n, k}$ be the standard probability Haar measure on $G_{n, k}$.
The projection $P_H f\colon H\to[0,+\infty)$
of a function $f\in Q_0^n$ on a subspace $H\in G_{n, k}$
is defined by the equality
$$
P_H f(y) := \sup\limits_{z\in H^\bot}f(y+z),
$$
where $H^\bot$ is the orthogonal complement of $H$ (see \cite{KM}, \cite{BCF}).
We point out that in the case $f = I_K$, where $I_K$ is the indicator function of
a convex body $K\subset \mathbb{R}^n$,
one has $P_H[I_K] = I_{P_H(K)}$, where $P_H(K)$ is the  orthogonal projection
of a set $K$ on a subspace $H$.
It is also worth mentioning that the bound \eqref{eq-exp}
and the upper semicontinuity assumption
imply that the supremum in the definition of $P_H f$ is attained, i.e.
$P_H f(y) = \max\limits_{z \in H^{\perp}} f(y + z)$.
We note that, for every function $f\in Q_0^n$, one has
$$
A_t(P_H f) = P_H\bigl(A_t(f)\bigr)\quad \forall t\in(0, +\infty).
$$
Along with projection of a function $f \in Q_0^n$,
one can also define its section
$S_H f: H \rightarrow [0, +\infty)$ by a subspace $H\in G_{n, k}$
by the natural equality $S_H f := f|_H$, i.e.
$$
S_H f(y) := f(y)\quad \forall y\in H.
$$
We note that in the case $f = I_K$ where
$K\subset \mathbb{R}^n$ is a convex body, one has $S_H[I_K] = I_{S_H(K)}$
where $S_H(K)$ is the section of a set $K$ by a subspace $H$.
Finally, for a nondegenerate linear mapping $A\colon \mathbb{R}^n\to \mathbb{R}^n$,
we define the linear image $Af\colon \mathbb{R}^n\to\mathbb{R}$
of a function $f\colon \mathbb{R}^n\to\mathbb{R}$ by the equality $Af(x):=f(A^{-1}x)$.
For an indicator function $f=I_K$, this definition
returns the indicator function of the linear image of the set, i.e. $AI_K = I_{AK}$.

As it was proved in \cite{BCF}, for logarithmically concave functions,
one has the following analog of the classical Steiner formula:
\begin{equation}\label{eq-stain}
\int_{\mathbb{R}^n}f_\delta(x)\, dx = \sum_{j=0}^nC_n^j W_j(f)\delta^j,
\end{equation}
where
$f_\delta(x):=\sup\limits_{|x-y|\le \delta}f(y)$, $C_n^j:=\frac{n!}{j!(n-j)!}$,
and where
\begin{multline}\label{eq-querm}
W_{n-k}(f):=\frac{\omega_n}{\omega_k}
\int_{0}^{+\infty} \int_{Gr_{n,k}} \lambda_k\bigl(P_H(A_t(f)\bigr) \, \nu_{n, k}(dH) \, dt
\\=\frac{\omega_n}{\omega_k}
\int_{Gr_{n,k}} \int_{0}^{+\infty} \lambda_k\bigl(A_t(P_H f)\bigr) \, dt\, \nu_{n, k}(dH)
= \frac{\omega_n}{\omega_k}
\int_{Gr_{n,k}} \|P_H f\|_{L^1(H)}\, \nu_{n, k}(dH).
\end{multline}

We recall that a function $f\in L^1(\mathbb{R}^n)$
belongs to the class $BV(\mathbb{R}^n)$ of functions with bounded variation
(see \cite[Definition 3.1]{AFP}) if,
for every $\theta \in \mathbb{R}^n$,
there is a bounded Borel measure
$D_\theta f$ such that
$$
\int_{\mathbb{R}^n}f(x)\partial_\theta\varphi(x)\, dx=
-\int_{\mathbb{R}^n}\varphi(x)\, D_\theta f(dx)\quad
\forall \varphi\in C_0^\infty(\mathbb{R}^n).
$$
For a function $f\in BV(\mathbb{R}^n)$, its variation
$V(f)=V(f, \mathbb{R}^n)$ is the following quantity
$$
V(f):=\sup\Bigl\{
\int_{\mathbb{R}^n}f(x){\rm div}\Phi(x)\, dx\colon
\Phi\in C_0^\infty(\mathbb{R}^n; \mathbb{R}^n),
\||\Phi|\|_\infty\le1\Bigr\}.
$$
We note that for a Sobolev function $f\in W^{1, 1}(\mathbb{R}^n)$,
one has
$$
V(f) = \int_{\mathbb{R}^n}|\nabla f(x)|\, dx.
$$
In \cite{Kr} (see the proof of Theorem $1$)
the inclusion $Q_0^n\subset BV(\mathbb{R}^n)$ was established.
Moreover, one has
$$
\|D_\theta f\|_{\rm TV} =
2|\theta|\cdot\|P_{\langle\theta\rangle^\bot}f\|_{L^1(\langle\theta\rangle^\bot)}\quad
\forall f\in Q_0^n
$$
where $\|\cdot\|_{\rm TV}$ is the total variation norm of a Borel measure.
We note that, for a function $f\in BV(\mathbb{R}^n)$, one has
$$
V(f) = \frac{1}{2\omega_{n-1}}\int_{S^{n-1}} \|D_\theta f\|_{\rm TV}\, \sigma_{n-1}(d\theta)
$$
where $S^{n-1}$ is the unit sphere in $\mathbb{R}^n$, and where $\sigma_{n-1}$
is the standard surface measure on $S^{n-1}$.
For a function $f\in Q_0^n$, the above equality can be also interpreted
as follows
$$
V(f)= \frac{1}{\omega_{n-1}}\int_{S^{n-1}}
\|P_{\langle\theta\rangle^\bot}f\|_{L^1(\langle\theta\rangle^\bot)}\, \sigma_{n-1}(d\theta)
=
\frac{n\omega_n}{\omega_{n-1}}
\int_{Gr_{n,n-1}} \|P_H f\|_{L^1(H)}\, \nu_{n, n-1}(dH)
= n W_1(f).
$$
We recall (see \cite[Theorem 3.47]{AFP}) that one has the embedding
$BV(\mathbb{R}^n)\subset L^{\frac{n}{n-1}}(\mathbb{R}^n)$. Moreover, for every function
$f\in BV(\mathbb{R}^n)$, the Sobolev inequality states that
\begin{equation}\label{eq-sob}
n\omega_n^{\frac{1}{n}}\|f\|_{\frac{n}{n-1}}\le V(f).
\end{equation}

In \cite{A-GMJR} (see also \cite{IN}), the classical notion of the John ellipsoid of a convex body
has been generalized to the case of logarithmically concave functions.
In particular, in the above mentioned paper, the following theorem has been proved.
\begin{theorem}\label{john}\
Let $f\in Q_0^n$. There is a unique function $\mathcal{E}(f)$
of the form $a\|f\|_\infty I_{\mathcal{E}}$ where
$\mathcal{E}$ is an ellipsoid  (i.e. an affine image of the unit ball)
and where $a\in[e^{-n}, 1]$,
such that
$$
1)\ \mathcal{E}(f)\le f;\quad
2)\ \int_{\mathbb{R}^n} \mathcal{E}(f)\, dx
=\max\Bigl\{\int_{\mathbb{R}^n} c I_{\mathcal{E}}\, dx\colon
c I_{\mathcal{E}}\le f, \mathcal{E} - \hbox{ ellipsoid}, c\in[0,+\infty)\Bigr\}.
$$
\end{theorem}

The function $\mathcal{E}(f)$ is called the John ellipsoid of the function $f\in Q_0^n$.
In \cite{A-GMJR} (see Corollary $1.3$ there) the following bound was established
\begin{equation}
\label{eq10}
I.rat(f) :=
\Bigl( \frac{\|f\|_1}{\|\mathcal{E}(f)\|_1} \Bigr)^{\frac{1}{n}} \le \frac{e}{n} (n!)^{\frac{1}{n}} \cdot I.rat(I_{\Delta^n}),
\end{equation}
where $\Delta^n$ is a regular $n$-dimensional simplex.

We recall the formulas for a volume and a radius of an insphere of a regular simplex with a side length $a$
$$V_n = \frac{a^n}{n!} \sqrt{\frac{n+1}{2^n}}, \quad r_n = \frac{a}{\sqrt{2n (n+1)}}.$$
Thus,
$$ V_n = \sqrt{\frac{n+1}{2^n}} \frac{(2n (n+1))^{\frac{n}{2}}}{n!} r_n^n = \frac{\sqrt{(n+1)} (n (n+1))^{\frac{n}{2}}}{n!} r_n^n.$$
We note that for the indicator functions the John ellipsoid is the indicator of the classical John ellipsoid, i.e. $\mathcal{E}(I_K) = I_{\mathcal{E}(K)}$. Applying the inequality \eqref{eq10} we get
\begin{equation}\label{eq-i-rat}
I.rat(f) =
\Bigl( \frac{\|f\|_1}{\|\mathcal{E}(f)\|_1} \Bigr)^{\frac{1}{n}}
\le \frac{e}{\sqrt{n}}\Bigl(\frac{(n+1)^{\frac{n+1}{2}}}{\omega_n}\Bigr)^{\frac{1}{n}}
\le \frac{e}{\sqrt{n}}\Bigl(\frac{(n+1)^{n + \frac{1}{2}} }{(2\pi)^{\frac{n}{2}}}\Bigr)^{\frac{1}{n}}
\le 4 \sqrt{n}
\end{equation}
where the last inequality holds since the function
$\frac{(x+1)^{1 + \frac{1}{2x}}}{x}$ is monotonically decreasing for $x \ge 1$.

\vskip .1in

We will also need the following functional analog of classical Alexandrov's inequality
(see \cite{BZ}) from \cite{BCF} (see Theorem $6.2$ there).

\begin{theorem}[The functional analog of Alexandrov's inequality]
\label{alexandrov}
Let $f\in Q_0^n$ and let $k, j$ be integers such that $0\le j<k\le n-1$.
Then
$$
\omega_n^{p-1}W_j(f^p)\le\bigl[W_k(f)\bigr]^p,
$$
where $p:=\frac{n-j}{n-k}$.
In particular,
$$
\omega_n^{\frac{n-k}{n}}\|f\|_{\frac{n}{k}}\le W_{n-k}(f),
$$
and, equivalently,
$$
\|f\|_{\frac{n}{k}} \le \frac{\omega_n^{\frac{k}{n}}}{\omega_k}
\int_{Gr_{n,k}} \|P_Hf\|_{L^1(L)}\, \nu_{n, k}(dH).
$$
\end{theorem}

\section{Functional Shephard problem of codimension one}
\label{sect-Sh-1}

The main result of this section is the following theorem.

\begin{theorem}\label{T1}
Let
$f_1\in BV(\mathbb{R}^n)$ and let $f_2\in Q_0^n$.
Assume that
\begin{equation}\label{cond-1}
\|D_\theta f_1\|_{\rm TV}\le \|D_\theta f_2\|_{\rm TV}\quad
\forall\theta\in \mathbb{R}^n,\ |\theta|=1.
\end{equation}
Then
$$
\|f_1\|_{\frac{n}{n-1}} \le 8 \sqrt{n} \|f_2\|_{\infty}^{\frac{1}{n}} \|f_2\|_1^{\frac{n-1}{n}}.
$$
\end{theorem}

\begin{proof}
Note that, for a linear image $Af_1$
under any nondegenerate linear mapping $A\colon \mathbb{R}^n\to \mathbb{R}^n$,
one has the equality
$$
\|D_\theta Af_1\|_{\rm TV} = |\det A|\|D_{A^{-1}\theta} f_1\|_{\rm TV}.
$$
A similar equality also holds for the function $f_2$.
Thus, the condition
\eqref{cond-1} is also valid for any linear images $Af_1$ and $Af_2$.
Moreover,
$$
\|Af_1\|_{\frac{n}{n-1}} = |\det A|^{\frac{n-1}{n}}\|f_1\|_{\frac{n}{n-1}};\quad
\|Af_2\|_{\infty}^{\frac{1}{n}} \|Af_2\|_1^{\frac{n-1}{n}}
=
|\det A|^{\frac{n-1}{n}}
\|f_2\|_{\infty}^{\frac{1}{n}} \|f_2\|_1^{\frac{n-1}{n}}.
$$
Therefore, it is sufficient to prove the theorem only for some linear images
$Af_1$, $Af_2$ of functions $f_1$ and $f_2$.

Let $\mathcal{E}(f_2)=a\|f_2\|_\infty I_{\mathcal{E}}$ be
the John ellipsoid of the function $f_2$ (see Theorem \ref{john}).
Without loss of generality we can assume that the center of the ellipsoid
$\mathcal{E}$ is at the origin (see Lemma 2.3 in \cite{A-GMJR}).
Let $A$ be a nondegenerate linear transformation that
maps the ellipsoid $\mathcal{E}$ onto the unit ball.
Then the function $a\|f_2\|_\infty I_{A\mathcal{E}} = \mathcal{E}(Af_2)$
is the John ellipsoid of the function $Af_2$.
In particular,
the John ellipsoid of the function $Af_2$ is of the form
$a\|Af_2\|_\infty I_{B^n_2}$
where $B^n_2$ is the unit ball in $\mathbb{R}^n$ centered at the origin.

Now, it is sufficient to prove our theorem only in the case when
the John ellipsoid $\mathcal{E}(f_2)$ of the function $f_2$ is of the form $a\|f_2\|_\infty I_{B^n_2}$.
By the Sobolev inequality \eqref{eq-sob}
\begin{multline*}
n \omega_n^{\frac{1}{n}} \|f_1\|_{\frac{n}{n-1}} \le
V(f_2) =
\frac{1}{2\omega_{n-1}}\int_{S^{n-1}} \|D_\theta f_1\|_{\rm TV}\, \sigma_{n-1}(d\theta)
\\ \le
\frac{1}{2\omega_{n-1}}\int_{S^{n-1}} \|D_\theta f_2\|_{\rm TV}\, \sigma_{n-1}(d\theta)
= nW_1(f_2).
\end{multline*}
Applying the functional analog of the Steiner formula \eqref{eq-stain}
we get that
$$
nW_1(f_2) = \lim\limits_{\delta\to +0}
\delta^{-1}\Bigl(\int_{\mathbb{R}^n}(f_2)_\delta(x)\, dx
- \int_{\mathbb{R}^n}f_2(x)\, dx\Bigr)
$$
where
$(f_2)_\delta(x):=\sup\limits_{|x-y|\le \delta}f(y)$.
Note that
\begin{multline*}
(f_2)_\delta(x) =
\sup\bigl\{f_2(y) [I_{B_2^n}(z)]^{\delta} \colon y = x - \delta z \bigr\}
\le \sup\bigl\{f_2(y) \bigl(\tfrac{f_2(z)}{a\|f_2\|_\infty}\bigr)^{\delta} \colon
y = x - \delta z \bigr\}
\\\le
a^{-\delta}\|f_2\|_\infty^{-\delta}
\sup\bigl\{f_2\bigl(\tfrac{y}{1+\delta} + \tfrac{\delta z}{1+\delta} \bigr)^{1+\delta}
\colon y = x - \delta z \bigr\} =
a^{-\delta}\|f_2\|_\infty^{-\delta} f_2\bigl(\tfrac{x}{1+\delta} \bigr)^{1+\delta}
\le e^{n\delta}f_2\bigl(\tfrac{x}{1+\delta}\bigr)
\end{multline*}
where, in the second inequality, we have used the assumption concerning the
John ellipsoid of the function $f_2$, in the third inequality,
we have used the logarithmic concavity of the function $f_2$, and
in the last inequality, we have used the bound for the number $a$ from theorem \ref{john}.
Thus,
\begin{multline*}
nW_1(f_2)\le \lim\limits_{\delta\to +0}
\delta^{-1}\Bigl(e^{n\delta}\int_{\mathbb{R}^n}f_2(\tfrac{x}{1+\delta})\, dx
- \int_{\mathbb{R}^n}f_2(x)\, dx\Bigr)
\\ =
\lim\limits_{\delta\to +0}
\delta^{-1}\bigl(e^{n\delta}(1+\delta)^n - 1\bigr)\int_{\mathbb{R}^n}f_2(x)\, dx
= 2n\int_{\mathbb{R}^n}f_2(x)\, dx
= 2n\|f_2\|_1^{\frac{n-1}{n}}\|f_2\|_1^{\frac{1}{n}}.
\end{multline*}
Applying the inequality \eqref{eq-i-rat} we obtian the estimate
$$
\|f_2\|_1^{\frac{1}{n}}
\le I.rat(f_2)\cdot \|\mathcal{E}(f_2)\|_1^{\frac{1}{n}}\le
4\sqrt{n}\cdot\|\mathcal{E}(f_2)\|_1^{\frac{1}{n}} \le
4\sqrt{n}\|f_2\|_\infty^{\frac{1}{n}}\omega_n^{\frac{1}{n}}.
$$
Therefore, we get the inequality
$$
n \omega_n^{\frac{1}{n}} \|f_1\|_{\frac{n}{n-1}} \le
2n\|f_2\|_1^{\frac{n-1}{n}}\cdot4\sqrt{n}\|f_2\|_\infty^{\frac{1}{n}}\omega_n^{\frac{1}{n}},
$$
which is equivalent to the announced bound.
The theorem is proved.
\end{proof}

\begin{corollary}\label{cor2}
Let
$f_1\in BV(\mathbb{R}^n)$ and let $f_2\in Q_0^n$.
Assume that
$$
\|D_\theta f_1\|_{\rm TV}\le \|D_\theta f_2\|_{\rm TV}\quad
\forall\theta\in \mathbb{R}^n,\ |\theta|=1.
$$
Then
$$
\|f_1\|_{\frac{n}{n-1}} \le 8e \sqrt{n} \|f_2\|_{\frac{n}{n-1}}$$
If, in addition, we assume that $f_1$ is logarithmically concave, then
$$
\|f_1\|_{\infty}^{\frac{1}{n}} \|f_1\|_1 \le 8e \sqrt{n} \|f_2\|_{\infty}^{\frac{1}{n}} \|f_2\|_1.
$$
\end{corollary}

\begin{proof}
Let $\|f_2\|_\infty = f(x_0)$.
We note that
$$
f_2(x)^\frac{n-1}{n}f_2(x_0)^\frac{1}{n}\le f_2(\tfrac{n-1}{n}x+\tfrac{1}{n}x_0)
$$
due to the logarithmic concavity of $f_2$.
From this bound we get that
\begin{multline*}
\|f_2\|_1^{\frac{n-1}{n}}\le
\Bigl(\int_{\mathbb{R}^n}[f_2(\tfrac{n-1}{n}x+\tfrac{1}{n}x_0)]^\frac{n}{n-1}
[f_2(x_0)]^{-\frac{1}{n-1}}\, dx\Bigr)^{\frac{n-1}{n}}
\\= (\tfrac{n}{n-1})^{n-1}\|f_2\|_{\frac{n}{n-1}}\|f_2\|_{\infty}^{-\frac{1}{n}}
\le e\|f_2\|_{\frac{n}{n-1}}\|f_2\|_{\infty}^{-\frac{1}{n}}.
\end{multline*}
Now, the first announced bound follows from Theorem \ref{T1}.

If we also assume that $f_1$ is logarithmically concave,
then, similarly to the bound above, we have
$$
e^{-1}\|f_1\|_1^{\frac{n-1}{n}}\|f_1\|_{\infty}^{\frac{1}{n}}\le \|f_2\|_{\frac{n}{n-1}}
$$
and the second announced bound again follows from Theorem \ref{T1}.
The corollary is proved.
\end{proof}

\begin{remark}
{\rm
We point out that, in the end of the proof of Theorem \ref{T1},
for a function $f_2 \in Q_0^n$ in John's position, we have actually proved the estimate
$$V(f_2) \le 8 n^{\frac{3}{2}} \omega_n^{\frac{1}{n}} \|f_2\|_1^{\frac{n-1}{n}} \|f_2\|_{\infty}^{\frac{1}{n}} \le c_1 n \|f_2\|_1^{\frac{n-1}{n}} \|f_2\|_{\infty}^{\frac{1}{n}} \le c_2 n \|f_2\|_{\frac{n}{n-1}}.$$
Recall that $V(f_2) = \|\nabla f_2\|_1$ for a function $f_2 \in W^{1,1}(\mathbb{R}^n)$.
Thus, the described bound is actually a reverse Sobolev inequality, which is an analog
of the reverse isoperimetric inequality from \cite{B291}}.
\end{remark}

\section{Functional Shephard problem: a general case}
\label{sect-Sh-2}

Firstly, we prove the following inequality, which, in a sense, is reverse
Alexandrov's inequality. We will use this bound further in this section.

\begin{lemma}
Let $f \in Q_0^n$ be in John's position (i.e. the John ellipsoid $\mathcal{E}(f) = a \|f\|_{\infty} I_{B_2^n}$). Then
$$W_{n-k}(f) \le W_{n-k-1}(f) \frac{n+k+1}{k+1}, \quad k = 0, ..., n-1.$$
\end{lemma}

\begin{proof}
1) We recall that $f_t(x):=\sup\limits_{|x-y|\le t} f(y)$.
We now consider the function $\varphi(t)\!=\!\|f_{t}\|_1$. We note that $f_{\delta + \rho} = (f_{\delta})_{\rho}$. Thus, applying the functional Steiner formula \eqref{eq-stain} we get that
$$\varphi(\delta + \rho) = \|f_{\delta + \rho}\|_1 = \sum_{k = 0}^n C_n^k W_k(f_{\delta}) \rho^k.$$
For any fixed $\delta$, the equality above is the Taylor series (polynomial) whose coefficients are the derivatives of the corresponding order, i.e.
$$
\varphi^{(k)}(t) = k! C_n^k W_k(f_{t}).
$$
From the identity $(f^{(k)})' = f^{(k+1)}$ we get
$$(k! C_n^k W_k(f_t))' = (k+1)! C_n^{k+1} W_{k+1}(f_t).$$
In other words, one has
$$(W_k(f_t))' = (k+1) \frac{C_n^{k+1}}{C_n^k} W_{k+1}(f_t) = (n-k) W_{k+1}(f_t),
\quad k = 0, ..., n-1.$$
Replacing  $k$ with $n-k$ above, we get that
\begin{equation}
\label{eq4.1}
(W_{n-k}(f_t))' = k W_{n-k+1}(f_t), \quad k = 1, ..., n.
\end{equation}

2) We recall that (see \eqref{eq-querm})
$$W_{n-k}(f) = \frac{\omega_n}{\omega_k} \int_{G_{n,k}} \int_{H}
\max_{z \in H^{\perp}} f(y + z) \,dy\, \nu_{n,k}(dH).$$
We now bound $(W_{n-k}(f_t))'\big|_{t = 0}$ from above.
By the definition of the derivative,
\begin{multline}\label{eq4.2}
\frac{\omega_k}{\omega_n} (W_{n-k}(f_{t}))'\big|_{t = 0}
\\ =
\lim_{t \rightarrow 0+} \frac{1}{t}
\Bigl( \int_{G_{n,k}} \int_H\max_{z \in H^{\perp}} f_{t}(y + z) \, dy \, \nu_{n,k}(dH)
- \int_{G_{n,k}} \int_H \max_{z \in H^{\perp}} f(y + z) \, dy \, \nu_{n,k}(dH) \Bigr).
\end{multline}
Similarly to the proof of Theorem \ref{T1}, we get
\begin{equation}\label{eq4.3}
f_{t}(z) \le e^{n t}f \bigl( \tfrac{z}{1 + t}\bigr).
\end{equation}
We point out that here we have used the assumption concerning the John's position of the function $f$.
Therefore,
\begin{equation}\label{eq4.3.1}
\int_H \max_{z \in H^{\perp}} f_{t}(y + z) \, dy
\leq \int_H \max_{z \in H^{\perp}} e^{n t}f \bigl( \tfrac{y + z}{1 + t}\bigr) \, dy.
\end{equation}
After the change of variables we get
\begin{equation}\label{eq4.4}
\int_H \max_{z \in H^{\perp}} e^{n t}f \bigl( \tfrac{y + z}{1 + t}\bigr) \, dy = (1 + t)^k e^{n t} \int_H \max_{z \in H^{\perp}} f(y + z) \, dy
\end{equation}

3) Finally, combining \eqref{eq4.2}, \eqref{eq4.3.1}, and \eqref{eq4.4}, we obtain that
\begin{multline*}
\frac{\omega_k}{\omega_n} (W_{n-k}(f_{t}))'\big|_{t = 0} \le \lim_{t \rightarrow 0+} \frac{1}{t} \bigl((1 + t)^k e^{n t} - 1\bigr) \int_{G_{n,k}} \int_H \max_{z \in H^{\perp}} f(y + z) \, dy \, \nu_{n,k}(dH) 
\\= (n + k) \int_{G_{n,k}} \int_H \max_{z \in H^{\perp}} f(y + z) \, dy \, \nu_{n,k}(dH) =
(n+k) \frac{\omega_k}{\omega_n} W_{n-k}(f).
\end{multline*}
From the equality \eqref{eq4.1} we get
$$k W_{n-k+1}(f) = (W_{n-k}(f_{t}))'\big|_{t = 0} \leq (n + k) W_{n-k}(f).$$
Therefore,
$$W_{n-k+1}(f) \leq W_{n-k}(f) \frac{n+k}{k}, \quad k = 1, ..., n,$$
and, replacing $k$ by $k+1$, we get the announced bound
$$W_{n-k}(f) \leq W_{n-k-1}(f) \frac{n+k+1}{k+1}, \quad k = 0, ..., n-1.$$
The lemma is proved.
\end{proof}

\begin{corollary}
\label{cor1}
Let $f \in Q_0^n$ be in John's position (i.e. the John ellipsoid $\mathcal{E}(f) = a \|f\|_{\infty} I_{B_2^n}$). Then
$$W_{n-k}(f) \le b_{n,k} \|f\|_1,$$
where
$$b_{n,k} = \frac{(n+k+1)...(2n)}{(k+1)...n}.$$
\end{corollary}
\begin{proof}
The estimate in the corollary is obtained by iterative application of the previous lemma since
$W_0(f) = \|f\|_1$.
\end{proof}
\begin{remark}
\label{remark2}
{\rm
We note that
$$b_{n,k} \le 4^{n-k}.$$
Indeed, let $s = n - k$. Then the sequence $b_{n,k} = b_{n,n-s}$ is monotonically
decreasing in $n$ for a fixed $s$ because
$$b_{n,n-s} = \frac{(2n-s+1)\ldots(2n)}{(n-s+1)\ldots n} = \prod_{j = 0}^{s-1} \Big(2 + \frac{j}{n - j} \Big)$$
where $2 + \frac{j}{n - j}$ is monotonically decreasing in $n$. Thus,
$$b_{n,n-s} \le b_{s,0} = C_{2s}^s \le 4^s,$$
and the inequality is proved.}
\end{remark}

\begin{remark}
{\rm
We point out the for a convex body $M$ in John's position, in place of the inequality~\ref{eq4.3},
one actually has a sharper bound $(I_M)_t(z)\le I_M(\frac{z}{1+t})$.
This bound implies monotonicity of the quermassintegrals of a convex body $K$
in John's position, i.e.
$$
W_{n-k}(M) \le W_{n-k-1}(M)\le W_0(M) = \lambda_n(M), \quad k = 0, ..., n-1.
$$
}
\end{remark}

\begin{lemma}
Let $f_1, f_2 \in Q_0^n$. The Shephard condition is invariant with respect to the
nondegenerate linear transformations, i.e. if
$$\|P_H f_1\|_{L^1(H)} \le \|P_H f_2\|_{L^1(H)} \quad \forall H \in G_{n,k},$$
then, $\forall A \in GL(n)$, one also has
$$\|P_H A f_1\|_{L^1(H)} \le \|P_H A f_2\|_{L^1(H)} \quad \forall H \in G_{n,k}.$$
\end{lemma}
\begin{proof}
Let $f \in Q_0^n$. We recall that $(Af)(x) := f(A^{-1} x)$. By the  definition
of the projection of a function, we have
$$
\|P_H A f\|_{L^1(H)} = \int_H \max_{z \in H^{\perp}} f(A^{-1} (y + z)) \, dy.
$$
We note that
$$\langle A^* y, A^{-1} z\rangle = \langle y, A A^{-1} z\rangle = \langle y,z \rangle = 0
\quad \forall\, y \in H, z \in H^{\perp}.$$
Therefore, $A^{-1} (H^{\perp}) \perp A^* H$.
By the equality of dimensions, we have
$A^{-1}(H^{\perp}) = (A^* H)^{\perp}$.
We now write $A^{-1} (y + z)$ as a sum of vectors from $A^* H$ and $(A^* H)^{\perp}$
as follows
$$A^{-1}(y+z) = (A^{-1} y - P_{(A^* H)^{\perp}} (A^{-1} y)) +
(P_{(A^* H)^{\perp}} (A^{-1} y) + A^{-1} z).$$
Here the first term $A^{-1} y - P_{(A^* H)^{\perp}} (A^{-1} y) = P_{A^* H} (A^{-1} y)$ is contained in $A^* H$ and the second term is contained in $(A^* H)^{\perp}$.
Since $A^{-1}(H^{\perp}) = (A^* H)^{\perp}$,
$$\int_H \max_{z \in H^{\perp}} f(A^{-1} (y + z)) \, dy = \int_H \max_{\widetilde{z} \in (A^*H)^{\perp}} f(P_{A^* H} (A^{-1} y) + \widetilde{z}) \, dy.$$
Let $S = P_{A^* H} \circ A^{-1}\colon H\to A^*H$.
We now show that this mapping is a bijection in order
to further make a change of variables in the integral.
Assume $Sy = 0$ for some $y \in H$, then $P_{A^* H} (A^{-1} y) = 0$ and $A^{-1} y \perp A^* H$
implying that $A^{-1} y \in (A^* H)^{\perp} = A^{-1} (H^{\perp})$. Therefore,
$y \in H^{\perp}$ and $y \in H$ implying that $y = 0$. So, $S$ is injective mapping.
The surjectivity of $S$ follows from its injectivity and
the equality of dimensions of subspaces $H$ and $A^* H$.
Finally, we can change the variables and obtain the equality
$$\int_H \max_{\widetilde{z} \in (A^*H)^{\perp}} f(S y + \widetilde{z}) \, dy = \frac{1}{|\det(S)|} \int_{A^* H} \max_{\widetilde{z} \in (A^*H)^{\perp}} f(\widetilde{y} + \widetilde{z}) \, d\widetilde{y}.$$
Thus,
$$\|P_H A f\|_{L^1(H)} = \frac{1}{|\det(S)|} \|P_{A^* H} f\|_{L^1(A^* H)}$$
where $S$ depends only on $A$ and $H$. So, the inequality
$$\|P_{A^* H} f_1\|_{L^1(A^* H)} \le \|P_{A^* H} f_2\|_{L^1(A^* H)}$$
implies
$$\|P_H A f_1\|_{L^1(H)} \le \|P_H A f_2\|_{L^1(H)}.$$
The lemma is proved.
\end{proof}

\begin{theorem}[The functional Shephard problem]\label{T2}
Let $f_1, f_2 \in Q_0^n$ and assume that the Shephard condition is satisfied:
\begin{equation}\label{eq-shep}
\|P_H f_1\|_{L^1(H)} \le \|P_H f_2\|_{L^1(H)} \quad \forall H \in G_{n,k}.
\end{equation}
Then
$$\|f_1\|_{\frac{n}{k}} \leq (16 \sqrt{n})^{n-k} \cdot \|f_2\|_1^{\frac{k}{n}} \cdot \|f_2\|_{\infty}^{\frac{n-k}{n}}.$$
\end{theorem}

\begin{proof}
We note that the assumption \eqref{eq-shep} and the assertion of the theorem do not change under the action of any nondegenerate affine mapping (see the lemma above) and after multiplication of
both functions $f_1$, $f_2$ by a positive constant.
Since $\mathcal{E}(Af_2) = A\mathcal{E}(f_2)$
for any non-degenerate affine mapping $A$ (see Lemma $2.3$ in \cite{A-GMJR}),
without loss of generality, we may assume that $f_2$ is in John's position
(i.e. the John ellipsoid $\mathcal{E}(f_2) = a \|f_2\|_{\infty} I_{B_2^n}$).
By functional Alexandrov's inequality from Theorem \ref{alexandrov}, we get
$$\|f_1\|_{\frac{n}{k}} \le \frac{\omega_n^{\frac{k}{n}}}{\omega_k} \int_{G_{n,k}} \|P_H f_1\|_{L^1(H)}\, \nu_{n, k}(dH).$$
Applying the Shephard condition \eqref{eq-shep} we get the inequality
$$\frac{\omega_n^{\frac{k}{n}}}{\omega_k} \int_{G_{n,k}} \|P_H f_1\|_{L^1(H)}\, \nu_{n, k}(dH) \le \frac{\omega_n^{\frac{k}{n}}}{\omega_k} \int_{G_{n,k}} \|P_H f_2\|_{L^1(H)}\, \nu_{n, k}(dH) = \omega_n^{-\frac{n-k}{n}} W_{n-k}(f_2).$$
By Corollary \ref{cor1}
$$
W_{n-k}(f_2) \le b_{n,k} \|f_2\|_1.
$$
Therefore,
$$\|f_1\|_{\frac{n}{k}} \le \omega_n^{-\frac{n-k}{n}} b_{n,k} \|f_2\|_1.$$
By the inequality \eqref{eq-i-rat}
$$\|f_2\|_1^{\frac{1}{n}} = (a \|f_2\|_{\infty} \omega_n)^{\frac{1}{n}} I.rat(f_2) \le \|f_2\|_{\infty}^{\frac{1}{n}} \omega_n^{\frac{1}{n}} 4 \sqrt{n}.$$
Thus,
$$\|f_1\|_{\frac{n}{k}} \le \omega_n^{-\frac{n-k}{n}} b_{n,k} \|f_2\|_1^{\frac{k}{n}} \|f_2\|_1^{\frac{n-k}{n}} \le (4 \sqrt{n})^{n-k} b_{n,k} \|f_2\|_1^{\frac{k}{n}} \|f_2\|_{\infty}^{\frac{n-k}{n}}.$$
Applying Remark \ref{remark2} we get the estimate
$$\|f_1\|_{\frac{n}{k}} \le (16 \sqrt{n})^{n-k} \|f_2\|_1^{\frac{k}{n}} \|f_2\|_{\infty}^{\frac{n-k}{n}}.$$
The theorem is proved.
\end{proof}

\begin{remark}
{\rm
The inequality obtained above is equivalent to the following one
$$\int_{\mathbb{R}^n} f_1^{\frac{n}{k}} \, dx \leq (16 \sqrt{n})^{(n-k) \frac{n}{k}} \|f_2\|_{\infty}^{\frac{n-k}{k}} \cdot \int_{\mathbb{R}^n} f_2 \, dx.$$
In the case when $f_1 = I_A$, $f_2 = I_B$, where $A$, $B$ are convex bodies
satisfying the Shephard condition, we get
\begin{equation}\label{eq-b-shep}
\lambda_n(A) \le (16 \sqrt{n})^{(n-k) \frac{n}{k}} \lambda_n(B).
\end{equation}
In particular, for a fixed codimension $s = n-k$, we have
$$(16 \sqrt{n})^{\frac{s (k + s)}{k}} = (16 \sqrt{n})^{s} \cdot (16 \sqrt{s + k})^{\frac{s^2}{k}}.$$
The second factor above tends to $1$ when $k$ tends to $\infty$. Therefore,
$$\lambda_n(A) \le C_s n^{s/2} \lambda_n(B)$$
where $C_s$ is a positive constant dependent only on $s$.}
\end{remark}

\begin{remark}
{\rm
In \cite{GK} (see Theorems $1.7$, $1.8$ there) the following two bounds were obtained.
\begin{theorem}[A. Giannopoulus, A. Koldobsky]\label{th-GK}
Let $A$ and $B$ be two convex bodies in $\mathbb{R}^n$ satisfying Shephard condition
$$|P_H(A)| \le |P_H(B)| \quad \forall H \in G_{n,k}.$$
Then
$$\lambda_n(A)^{\frac{1}{n}} \le c_1 \sqrt{\frac{n}{k}} \log\Bigl(\frac{e n}{k}\Bigr) \lambda_n(B)^{\frac{1}{n}},$$
$$\lambda_n(A)^{\frac{1}{n}} \le c_2 \log(n) \lambda_n(B)^{\frac{1}{n}}$$
where $c_1$ and $c_2$ are some absolute positive constants.
\end{theorem}
We now compare our result with these estimates.
From the inequality \eqref{eq-b-shep} follows that
$$
\frac{\lambda_n(A)^{\frac{1}{n}}}{\lambda_n(B)^{\frac{1}{n}}} \le (16 \sqrt{n})^{\frac{s}{n-s}}
\to 1
$$
when $ n \to \infty$ and the codimension $s$ is bounded by some absolute constant $M$.
Thus, there is a positive constant $C$,
depending only on $M$, such that
$$
\lambda_n(A)^{\frac{1}{n}} \le C \lambda_n(B)^{\frac{1}{n}}.
$$
The first inequality from Theorem \ref{th-GK}
also gives
$\lambda_n(A)^{\frac{1}{n}} \le C \lambda_n(B)^{\frac{1}{n}}$
when the codimension $s$ is bounded. However, Theorem \ref{th-GK} implies
a similar bound under the less restrictive assumption that $\frac{s}{k}$ is bounded.
On the other hand,  for the quantity
$\tfrac{\lambda_n(A)}{\lambda_n(B)}$, both bounds in Theorem \ref{th-GK}
depend on the dimension $n$ exponentially,
while the inequality \eqref{eq-b-shep} provides
a polynomial in $n$ estimate when the codimension $s$ is bounded.
}
\end{remark}

Similarly to Corollary \ref{cor2} above, we have the following two bounds.
\begin{corollary}
Let $f_1, f_2 \in Q_0^n$ and assume that the Shephard condition \eqref{eq-shep} is satisfied.
Then
$$\|f_1\|_{\frac{n}{k}} \leq (16 e \sqrt{n})^{n-k} \cdot \|f_2\|_{\frac{n}{k}},$$
$$\|f_1\|_1^{\frac{k}{n}} \cdot \|f_1\|_{\infty}^{\frac{n-k}{n}} \le (16 e \sqrt{n})^{n-k} \|f_2\|_1^{\frac{k}{n}} \cdot \|f_2\|_{\infty}^{\frac{n-k}{n}}$$
\end{corollary}

\begin{proof}
The corollary follows from Theorem \ref{T2} and the lemma below.
\end{proof}

\begin{lemma} \label{lemma3}
Let $f \in Q_0^n$. Then
$$\|f\|_1^{\frac{k}{n}} \cdot \|f\|_{\infty}^{\frac{n-k}{n}} \le \Bigl(\frac{n}{k}\Bigr)^{k} \|f\|_{\frac{n}{k}} \le e^{n-k} \|f\|_{\frac{n}{k}}.$$
\end{lemma}
\begin{proof}
Let $\|f\|_{\infty} = f(x_0)$. We note that
$$f(x)^{\frac{k}{n}} \cdot f(x_0)^{\frac{n-k}{n}} \le f(\tfrac{k}{n}x + \tfrac{n-k}{n}x_0)$$
due to the logarithmic concavity of $f$. From this bound, we get that
\begin{multline*}
\|f\|_1^{\frac{k}{n}}\le
\Bigl(\int_{\mathbb{R}^n}[f(\tfrac{k}{n}x+\tfrac{n-k}{n}x_0)]^\frac{n}{k}
[f(x_0)]^{-\frac{n-k}{k}}\, dx\Bigr)^{\frac{k}{n}}
\\= (\tfrac{n}{k})^{k} \|f\|_{\frac{n}{k}} \|f\|_{\infty}^{-\frac{n-k}{n}}
= (1 + \tfrac{n-k}{k})^{k} \|f\|_{\frac{n}{k}} \|f\|_{\infty}^{-\frac{n-k}{n}}
\le e^{n-k} \|f\|_{\frac{n}{k}} \|f\|_{\infty}^{-\frac{n-k}{n}}
\end{multline*}
and the lemma is proved.
\end{proof}

\section{Functional Busemann--Petty problem}
\label{sect-BP}
In this and the next sections we will need the
functional analog of the following classical Grinberg's
inequality for convex bodies (see \cite{Grinb}):
Let $M$ be a convex body in $\mathbb{R}^n$, then, for every $k\in [1, n]\cap \mathbb{N}$, one has
\begin{equation}\label{grinberg}
\int_{G_{n,k}} \bigl[\lambda_k(M \cap H)\bigr]^n \, \nu_{n, k}(dH)\le
\frac{\omega_k^n}{\omega_n^k} \cdot \bigl[\lambda_n(M)\bigr]^k.
\end{equation}

\begin{theorem}[The functional analog of Grinberg's inequality]
\label{funcgrinberg}
Let $f\in Q_0^n$. Then, for every $k\in [1, n]\cap \mathbb{N}$,
one has
$$
\int_{G_{n,k}} \|S_H f\|_{L^1(H)}^n \, \nu_{n, k}(dH)\le
\frac{\omega_k^n}{\omega_n^k} \cdot \|f\|_1^{k}\cdot \|f\|_{\infty}^{n-k} .
$$
\end{theorem}

\begin{proof}
By the definition
$$\|S_H f\|_{L^1(H)} = \int_H f(y)\, dy.$$
We recall that $A_t(f):=\{x\in \mathbb{R}^n\colon f(x)\ge t\}$.
Applying Fubini's theorem and the generalized Minkowski inequality we get
\begin{multline*}
\int_{G_{n,k}} \Bigl(\int_H f(y)\, dy\Bigr)^n \, \nu_{n, k}(dH)
= \int_{G_{n,k}} \Bigl(\int_{0}^{+\infty} \lambda_k\bigl(A_t(f) \cap H\bigr) \, dt\Bigr)^n \, \nu_{n, k}(dH)
\\ \le
\Bigl(\int_{0}^{+\infty}
\Bigl(\int_{G_{n,k}}\bigl[\lambda_k\bigl(A_t(f) \cap H\bigr)\bigr]^n\, \nu_{n, k}(dH)\Bigr)^{1/n} \, dt\Bigr)^n.
\end{multline*}
By Grinberg's
inequality \eqref{grinberg} the last expression does not exceed
\begin{multline*}
\frac{\omega_k^n}{\omega_n^k} \Bigl(\int_{0}^{+\infty} \bigl[\lambda_n\bigl(A_t(f)\bigr)\bigr]^{k/n}
\, dt \Bigr)^n
=
\frac{\omega_k^n}{\omega_n^k} \Bigl(\int_{0}^{\|f\|_{\infty}} \bigl[\lambda_n\bigl(A_t(f)\bigr)\bigr]^{k/n} \, dt \Bigr)^n
\\ \le
\frac{\omega_k^n}{\omega_n^k} \Bigl(\int_{0}^{\|f\|_{\infty}} \lambda_n\bigl(A_t(f)\bigr) \, dt \Bigr)^k
\Bigl(\int_{0}^{\|f\|_{\infty}} 1 \, dt \Bigr)^{n - k} =
\frac{\omega_k^n}{\omega_n^k} \cdot \|f\|_1^k \cdot \|f\|_{\infty}^{n-k}
\end{multline*}
where the second bound follows from Holder's inequality with the parameters
$p = \frac{n}{k}$ and $q = \frac{n}{n - k}$,
and the last equality follows from Fubini's theorem.
The theorem is proved.
\end{proof}

We recall the definition of the isotropic constant of a log-concave function.
\begin{definition}[see Definition $2.3.11$ in \cite{BGVV}]
\label{def1}
Let $f \in Q_0^n$ and let ${\rm Cov}(f)$ be the covariance matrix of the measure
$\frac{1}{\|f\|_1} f(x) \, dx$. The number
$$L_f = \Big( \frac{\|f\|_{\infty}}{\|f\|_1} \Big)^{\frac{1}{n}} \big(\det {\rm Cov}(f)\big)^{\frac{1}{2n}}$$
is called the isotropic constant of the function $f$.
\end{definition}
We note that the isotropic constant $L_f$ does not change after
nondegenerate linear transformations
and after multiplication of the function $f$ by a positive number.

\begin{definition}[see Definition $2.3.9$ in \cite{BGVV}]
We say that the function $f \in Q_0^n$ is in the isotropic position if

$1) \ \|f\|_1 = 1,$

$2) \ {\rm Cov}(f) = Id$ --- identity matrix,

$3) \ \displaystyle \int_{\mathbb{R}^n} x \, f(x) \, dx = 0$ --- the function is centered.
\end{definition}
It is known that, for every function $f \in Q_0^n$,
there is a number $\alpha > 0$ and
a nondegenerate linear transformation $A \in GL(n)$
such that the function $\alpha\cdot \bigl(A\bigl[f(\cdot + a)\bigr]\bigr)$ is in the isotropic position
where $a$ is the mean vector of $f$.

\begin{lemma}
\label{lemma1}
There is an absolute constant $C > 0$ such that,
for every log-concave probability density $f$ in the isotropic position,
for all subspaces $H \in G_{n,k}$, the following inequality holds
$$C^{n-k} \le \|S_H f \|_{L^1(H)} \le L_{n-k}^{n-k}$$
where $L_{d} = \sup  \{L_g\colon g \in Q_0^d\}$.
\end{lemma}
\begin{proof}
Let $X$ be a random vector with the distribution density $f$ and 
let $H$ be a $k$-dimensional subspace of $\mathbb{R}^n$. 
Let $\theta_1, ..., \theta_n$ be an orthonormal basis in $\mathbb{R}^n$ 
such that $\theta_1, ..., \theta_{n-k}$ is a 
basis in $H^{\perp}$ and $\theta_{n-k+1}, ..., \theta_n$ is a basis in $H$.

Let
$$Y =
\begin{pmatrix}
\langle X, \theta_1 \rangle \\
\dots \\
\langle X, \theta_{n-k} \rangle \\
\end{pmatrix},$$
i.e. $Y$ is the projection of the vector $X$ onto the subspace $H^{\perp}$.

We note that the distribution density of $Y$ is of the form
$$\varrho_Y(t_1, ..., t_{n-k}) = \int_H f(y + t_1 \theta_1 + ... + t_{n-k} \theta_{n-k}) \, dy.$$
In particular,
$$\varrho_Y(0) = \int_H f(y) \, dy.$$
Since the density $f$ is in the isotropic position, the density $\varrho_Y$
is also in the isotropic position. Firstly, by Fradelizi
inequality (see Theorem 2.2.2 in \cite{BGVV}) and
by the lower bound of the isotropic constant
(see Proposition 2.3.12 in \cite{BGVV}) we have
$$e \varrho_Y(0)^{\frac{1}{n-k}} \ge \|\varrho_Y\|_{\infty}^{\frac{1}{n-k}}
= L_{\varrho_Y} \ge c>0.$$
Thus,
$$\left(\int_{H} f(y) \, dy\right)^{\frac{1}{n-k}} = \varrho_Y(0)^{\frac{1}{n-k}} \geq \frac{c}{e}.$$
Secondly,
$$\int_H f(y) \, dy = \varrho_Y(0) \le \|\varrho_Y\|_{\infty} = L_{\varrho_Y}^{n-k} \le L_{n-k}^{n-k}.$$
The lemma is proved.
\end{proof}

\begin{lemma}
\label{lemma2}
There is an absolute constant $c_0 > 0$ such that, $\forall k \in \mathbb{N}$, $0 < k < n$,
the following inequality holds
$$\frac{\omega_k^n}{\omega_n^k} \le c_0^{n (n-k)}.$$
\end{lemma}

\begin{proof}
We note that the stated bound is equivalent to the following inequality
\begin{equation}
\label{eq1}
\frac{\omega_k^{k + s}}{\omega_{k + s}^k} \le c_0^{(k + s) s}
\end{equation}
where $k, s \in \mathbb{N}$. We prove it by induction in $s$.

1) The base case.
We prove the inequality \eqref{eq1} for every $k \in \mathbb{N}$ and $s = 1$.
\begin{equation*}
\frac{\omega_{k}}{\omega_{k+1}^{\frac{k}{k+1}}} =
\frac{\left( \Gamma\left(\frac{k+1}{2} + 1\right) \right)^{\frac{k}{k+1}}}{\Gamma\left(\frac{k}{2} + 1\right)}
\sim \frac{(\sqrt{\pi k})^{\frac{k}{k+1}} \left( \frac{k+1}{2e} \right)^{\frac{k}{2}}}{\sqrt{\pi k} \left( \frac{k}{2e} \right)^{\frac{k}{2}}}
= \frac{1}{(\sqrt{\pi k})^{\frac{1}{k+1}}} \Big( 1 + \frac{1}{k} \Big)^{\frac{k}{2}}
\sim \sqrt{e}.
\end{equation*}
Therefore, there is an absolute constant $c_0 > 0$ such
that
$$
\frac{\omega_{k}^{k+1}}{\omega_{k+1}^{k}} \le c_0^{k+1}.
$$

2) The inductive step.
Suppose that the inequality \eqref{eq1} holds for every $k \in \mathbb{N}$ and some
fixed $s \in \mathbb{N}$. We now prove the bound \eqref{eq1}
for every $k \in \mathbb{N}$ and $s+1$ in place of $s$. So,
we want to prove the estimate
$$\frac{\omega_k^{k+s+1}}{\omega_{k+s+1}^k}
\le c_0^{(k+s+1) (s+1)}.$$
By the induction hypothesis
$$\omega_k^{k+s+1} = (\omega_k^{k+s})^{\frac{k+s+1}{k+s}} \le (\omega_{k+s}^k c_0^{(k+s) s})^{\frac{k+s+1}{k+s}} = (\omega_{k+s}^{k+s+1})^{\frac{k}{k+s}} c_0^{(k+s+1)s}.$$
Applying the base of induction we get
$$(\omega_{k+s}^{k+s+1})^{\frac{k}{k+s}} c_0^{(k+s+1)s} \le (\omega_{k+s+1}^{k+s} c_0^{k+s+1})^{\frac{k}{k+s}} c_0^{(k+s+1)s} = \omega_{k+s+1}^k c_0^{(k+s+1)(s+\frac{k}{k+s})} \le \omega_{k+s+1}^k c_0^{(k+s+1)(s+1)}.$$
The lemma is proved.
\end{proof}

\begin{remark}
{\rm
We point out, that the exact value of the constant $c_0$ in the inequality above
was obtained in \cite[Lemma 2.1]{KL00}.
However, the simple estimate from Lemma \ref{lemma2} is sufficient for our purposes.
}
\end{remark}

\begin{theorem}[The functional Busemann--Petty problem] \label{T4}
There is an absolute constant $C\! >~\!\!\!0$ such that,
if $f_1, f_2 \in Q_0^n$ satisfy the Busemann--Petty condition
\begin{equation}\label{eq-bp}
\|S_H f_1\|_{L^1(H)} \le \|S_H f_2\|_{L^1(H)}\quad \forall H\in G_{n, k},
\end{equation}
and $f_1^{\frac{n}{k}}$ is centered, then the following inequality holds
$$\|f_1\|_{\frac{n}{k}} \le (C L)^{n-k} \cdot \|f_2\|_1^{\frac{k}{n}} \cdot \|f_2\|_{\infty}^{\frac{n-k}{n}}$$
where $L$ is the isotropic constant of the function $f_1^{\frac{n}{k}}$.
\end{theorem}

\begin{proof}
We note that the assumption \eqref{eq-bp} and the assertion of the theorem do not
change under the action of any nondegenerate linear transformation
and after multiplication of both functions $f_1$, $f_2$ by a positive constant, because
$$\|S_H A f\|_{L^1(H)} = \frac{1}{|\det(T)|} \|S_{A^{-1} H} f\|_{L^1(A^{-1} H)}$$
where $T = A^{-1}|_{H}: H \rightarrow A^{-1} H$. Therefore, we may consider only the case when $f_1^{\frac{n}{k}} \in Q_0^n$ is a probability density in the isotropic position. By Lemma \ref{lemma1}, for a function $f\in Q_0^n$ in the isotropic position,
$$C^{n-k} \le \|S_H f\|_{L^1(H)}.$$
Thus,
\begin{multline*}
C^{n-k} \le \Big( \int_{G_{n,k}} \|S_H f_1^{\frac{n}{k}}\|_{L^1(H)}^n \, \nu_{n,k}(dH) \Big)^{\frac{1}{n}}
= \Big( \int_{G_{n,k}} \|S_H (f_1^{\frac{n-k}{k}} \cdot f_1)\|_{L^1(H)}^n \, \nu_{n,k}(dH) \Big)^{\frac{1}{n}} \le \\
\le \|f_1\|_{\infty}^{\frac{n-k}{k}} \Big( \int_{G_{n,k}} \|S_H f_1\|_{L^1(H)}^n \, \nu_{n,k}(dH) \Big)^{\frac{1}{n}}.
\end{multline*}
Since $f_1^{\frac{n}{k}}$ is in the isotropic position,
$$L = L_{f_1^{\frac{n}{k}}} = \|f_1^{\frac{n}{k}}\|_{\infty}^{\frac{1}{n}} = \|f_1\|_{\infty}^{\frac{1}{k}}.$$
Using this equality  and
applying the Busemann--Petty condition~\eqref{eq-bp}, we get
$$\|f_1\|_{\infty}^{\frac{n-k}{k}} \Big( \int_{G_{n,k}} \|S_H f_1\|_{L^1(H)}^n \, \nu_{n,k}(dH) \Big)^{\frac{1}{n}} \le L^{n-k} \Big( \int_{G_{n,k}} \|S_H f_2\|_{L^1(H)}^n \, \nu_{n,k}(dH) \Big)^{\frac{1}{n}}.$$
By functional Grinberg's inequality from Theorem \ref{funcgrinberg}
$$ L^{n-k} \Big( \int_{G_{n,k}} \left( \|S_H f_2\|_{L^1(H)} \, dy \right)^n \, \nu_{n,k}(dH) \Big)^{\frac{1}{n}} \le L^{n-k} \cdot \frac{\omega_k}{\omega_n^{\frac{k}{n}}} \cdot \|f_2\|_1^{\frac{k}{n}} \cdot \|f_2\|_{\infty}^{\frac{n-k}{n}}.$$
Since, by Lemma \ref{lemma2},
$$\frac{\omega_k}{\omega_n^{\frac{k}{n}}} \le c_0^{n-k},$$
we obtain the bound
$$C^{n-k} \le (c_0 L)^{n-k} \cdot \|f_2\|_1^{\frac{k}{n}} \cdot \|f_2\|_{\infty}^{\frac{n-k}{n}}.$$
Finally, $f_1^{\frac{n}{k}}$ is a probability density, therefore, $\|f_1\|_{\frac{n}{k}} = 1$ and
$$\|f_1\|_{\frac{n}{k}} \le \Big(\frac{c_0 L}{C}\Big)^{n-k} \cdot \|f_2\|_1^{\frac{k}{n}} \cdot \|f_2\|_{\infty}^{\frac{n-k}{n}}.$$
The theorem is proved.
\end{proof}

The assumption that the function $f_1^{\frac{n}{k}}$ is centered seems rather odd.
In the following theorem, we present the result when the function $f_1$
is centered, which seems more natural.

\begin{theorem}\label{T5}
There is an absolute constant $C > 0$ such that,
if $f_1, f_2 \in Q_0^n$ satisfy the Busemann--Petty condition \eqref{eq-bp}
and $f_1$ is centered, then the following inequality holds
$$\|f_1\|_1^{\frac{k}{n}} \cdot \|f_1\|_{\infty}^{\frac{n-k}{n}} \le
(C L)^{n-k} \cdot \|f_2\|_1^{\frac{k}{n}} \cdot \|f_2\|_{\infty}^{\frac{n-k}{n}}$$
where $L$ is the isotropic constant of the function $f_1$.
\end{theorem}

\begin{proof}
We note that, like in the previous theorem, the assumption \eqref{eq-bp}
and the assertion of the theorem do not change under the
action of any nondegenerate linear mapping and after multiplication of
both functions $f_1$, $f_2$ by a positive constant.
Therefore, we may consider only the case when $f_1 \in Q_0^n$
is a probability density in the isotropic position. By Lemma \ref{lemma1},
for $f$ in the isotropic position,
$$C^{n-k} \le \|S_H f\|_{L^1(H)}.$$
Thus, applying the Busemann--Petty condition \eqref{eq-bp} and functional Grinberg's inequality from Theorem \ref{funcgrinberg} we get
\begin{multline*}
C^{n-k} \le \Big( \int_{G_{n,k}} \|S_H f_1\|_{L^1(H)}^n \, \nu_{n,k}(dH) \Big)^{\frac{1}{n}}
\le
\Big( \int_{G_{n,k}} \|S_H f_2\|_{L^1(H)}^n \, \nu_{n,k}(dH) \Big)^{\frac{1}{n}}
\le \\ \le \frac{\omega_k}{\omega_n^{\frac{k}{n}}} \cdot
\|f_2\|_1^{\frac{k}{n}} \cdot \|f_2\|_{\infty}^{\frac{n-k}{n}}\le
c_0^{n-k}\cdot \|f_2\|_1^{\frac{k}{n}} \cdot \|f_2\|_{\infty}^{\frac{n-k}{n}}
\end{multline*}
where we have applied Lemma \ref{lemma2} in the last inequality. Thus,
$$C^{n-k} \le c_0^{n-k} \|f_2\|_1^{\frac{k}{n}} \cdot \|f_2\|_{\infty}^{\frac{n-k}{n}}.$$
Since $f_1$ is in the isotropic position,
$$L = L_{f_1} = \|f_1\|_{\infty}^{\frac{1}{n}}.$$
Multiplying the last inequality by $L^{n-k}$ we get
$$C^{n-k} \|f_1\|_{\infty}^{\frac{n-k}{n}} \le (c_0 L)^{n-k} \|f_2\|_1^{\frac{k}{n}} \cdot \|f_2\|_{\infty}^{\frac{n-k}{n}}.$$
Finally, $f_1$ is a probability density, therefore, $\|f_1\|_1 = 1$ and
$$\|f_1\|_1^{\frac{k}{n}} \cdot \|f_1\|_{\infty}^{\frac{n-k}{n}} \le \Bigl( \frac{c_0 L}{C} \Bigr)^{n-k} \|f_2\|_1^{\frac{k}{n}} \cdot \|f_2\|_{\infty}^{\frac{n-k}{n}}.$$
The theorem is proved.
\end{proof}

\begin{remark}
{\rm
As it has been recently established by Klartag and Lehec in \cite{KL},
one always has $L_d\le C(1+\log d)^4$ for some absolute constant $C>0$.
Bourgain's slicing conjecture asserts that one always has $L_d\le C$ for some universal konstant
$C>0$. This conjecture is still an open problem.
}
\end{remark}

\section{Functional Milman problem}
\label{sect-M}

\begin{theorem}[The functional Milman problem]\label{T3}
Let $f_1, f_2 \in Q_0^n$ and assume that the Milman condition
\begin{equation}\label{eq-milm}
\|P_H f_1\|_{L^1(H)} \le \|S_H f_2\|_{L^1(H)}\quad \forall H\in G_{n, k}
\end{equation}
is satisfied.
Then
$$\|f_1\|_{\frac{n}{k}} \leq \|f_2\|_1^{\frac{k}{n}} \cdot \|f_2\|_{\infty}^{\frac{n-k}{n}}.$$
\end{theorem}

\begin{proof}
Applying the condition \eqref{eq-milm} and
functional Alexandrov's inequality from Theorem \ref{alexandrov} we get
\begin{multline*}
\|f_1\|_{\frac{n}{k}} \le \frac{\omega_n^{\frac{k}{n}}}{\omega_k} \int_{G_{n,k}} \|P_H f_1\|_{L^1(H)}\, \nu_{n, k}(dH)
\le\\\le
\frac{\omega_n^{\frac{k}{n}}}{\omega_k} \int_{G_{n,k}} \|S_H f_2\|_{L^1(H)}\, \nu_{n, k}(dH)
\le
\frac{\omega_n^{\frac{k}{n}}}{\omega_k} \Bigl(\int_{G_{n,k}}
\|S_H f_2\|_{L^1(H)}^n \, \nu_{n, k}(dH) \Bigr)^{\frac{1}{n}}.
\end{multline*}
By functional Grinberg's inequality from Theorem \ref{funcgrinberg},
the last expression does not exceed
$$ \frac{\omega_n^{\frac{k}{n}}}{\omega_k}
\Bigl(\frac{\omega_k^n}{\omega_n^k} \cdot \|f_2\|_1^{k}\cdot \|f_2\|_{\infty}^{n-k} \Bigr)^{\frac{1}{n}}.$$
Thus,
$$ \|f_1\|_{\frac{n}{k}} \le \|f_2\|_1^{\frac{k}{n}}\cdot \|f_2\|_{\infty}^{\frac{n-k}{n}},$$
and the theorem is proved.
\end{proof}

We now prove a functional analog of Theorem $1.3$ from \cite{GK}
which provides a bound in the Milman problem in terms
of the isotropic constant of one of the functions.

\begin{theorem}\label{T6}
Let $f_1, f_2 \in Q_0^n$ satisfy the
Milman condition \eqref{eq-milm} and let $f_2$ be
centered. Then the following inequality holds
$$
\|f_1\|_{\frac{n}{k}} \le \Bigl( \frac{L_{n-k}}{L_{f_2}} \Bigr)^{n-k} \|f_2\|_1^{\frac{k}{n}} \cdot \|f_2\|_{\infty}^{\frac{n-k}{n}}
$$
where $L_{d} = \sup\{L_g\colon g \in Q_0^d\}$.
\end{theorem}

\begin{proof}
Without loss of generality we can assume that $\|f_2\|_1 = 1$
(i.e. $f_2$ is probability destiny).
Let $A \in SL(n)$ be a linear transformation
such that ${\rm Cov} A f_2$ is a scalar matrix with a number $a$ on the diagonal.
According to the definition \ref{def1} of the isotropic constant, we have
$$L = L_{f_2} = L_{Af_2} = \Big( \frac{\|Af_2\|_{\infty}}{\|Af_2\|_1} \Big)^{\frac{1}{n}} \big(\det {\rm Cov}(Af_2)\big)^{\frac{1}{2n}} = \|f_2\|_{\infty}^{\frac{1}{n}} \sqrt{a}.$$
Let $g(x) = a^{\frac{n}{2}}\cdot A[f_2](\sqrt{a} x)$. Since the function $g$ is in the isotropic position, Lemma \ref{lemma1} implies
$$\|S_H g\|_{L^1(H)} \le L_{n-k}^{n-k}.$$
Therefore,
$$\Bigl(\int_{G_{n,k}}
\|S_H g\|_{L^1(H)}^n \, \nu_{n, k}(dH) \Bigr)^{\frac{1}{n}} \le L_{n-k}^{n-k}.$$
From the definition of the function $g$, we get the equality
$$\Bigl(\int_{G_{n,k}}
\|S_H g\|_{L^1(H)}^n \, \nu_{n, k}(dH) \Bigr)^{\frac{1}{n}} = a^{\frac{n-k}{2}} \Bigl(\int_{G_{n,k}}
\|S_H Af_2\|_{L^1(H)}^n \, \nu_{n, k}(dH) \Bigr)^{\frac{1}{n}}$$
We note that the quantity
$$\widetilde{\Phi}_{n-k} (f) = \frac{\omega_n}{\omega_k} \Bigl(\int_{G_{n,k}}
\|S_H f\|_{L^1(H)}^n \, \nu_{n, k}(dH) \Bigr)^{\frac{1}{n}}$$
is invariant under the transformations from $SL(n)$, i.e.
$\widetilde{\Phi}_{n-k}(f) = \widetilde{\Phi}_{n-k} (Af)$ $\forall A \in SL(n)$.
The proof of this fact almost verbatim repeats the arguments from the proof
of Theorem $1$ in \cite{Grinb}.
Thus,
$$\Bigl(\int_{G_{n,k}}
\|S_H f_2\|_{L^1(H)}^n \, \nu_{n, k}(dH) \Bigr)^{\frac{1}{n}} \le \Bigl( \frac{L_{n-k}}{\sqrt{a}} \Bigr)^{n-k} = \Bigl( \frac{L_{n-k}}{L} \Bigr)^{n-k} \|f_2\|_{\infty}^{\frac{n-k}{n}}.$$
Applying the condition \eqref{eq-milm}, functional Alexandrov's inequality from Theorem \ref{alexandrov},
and the last bound,
we get
\begin{multline*}
\|f_1\|_{\frac{n}{k}} \le \frac{\omega_n^{\frac{k}{n}}}{\omega_k} \int_{G_{n,k}} \|P_H f_1\|_{L^1(H)}\, \nu_{n, k}(dH)
\le
\frac{\omega_n^{\frac{k}{n}}}{\omega_k} \int_{G_{n,k}} \|S_H f_2\|_{L^1(H)}\, \nu_{n, k}(dH)
\le \\ \le
\frac{\omega_n^{\frac{k}{n}}}{\omega_k} \Bigl(\int_{G_{n,k}}
\|S_H f_2\|_{L^1(H)}^n \, \nu_{n, k}(dH) \Bigr)^{\frac{1}{n}} \le \frac{\omega_n^{\frac{k}{n}}}{\omega_k} \Bigl( \frac{L_{n-k}}{L} \Bigr)^{n-k} \|f_2\|_{\infty}^{\frac{n-k}{n}} \le \Bigl( \frac{L_{n-k}}{L} \Bigr)^{n-k} \|f_2\|_{\infty}^{\frac{n-k}{n}}.
\end{multline*}
The last estimate above follows from the inequality $\omega_n^k \le \omega_k^n$
which is equivalent to the bound
$$
\Gamma(\tfrac{k}{2} + 1)^n \le \Gamma(\tfrac{n}{2} + 1)^k.
$$
By the log-convexity of the gamma-function we actually have
$$\Gamma(\tfrac{k}{2} + 1) \le \Gamma(\tfrac{n}{2} + 1)^{\frac{k}{n}} \cdot \Gamma(1)^{\frac{n-k}{n}} = \Gamma(\tfrac{n}{2} + 1)^{\frac{k}{n}}.$$
The theorem is proved.
\end{proof}

\begin{remark}
{\rm
We point out that under the assumptions of the previous theorem, the following inequality has actually been obtained
$$\|f_1\|_{\frac{n}{k}} \le  \frac{\omega_n^{\frac{k}{n}}}{\omega_k} \Bigl( \frac{L_{n-k}}{L_{f_2}} \Bigr)^{n-k} \|f_2\|_1^{\frac{k}{n}} \cdot \|f_2\|_{\infty}^{\frac{n-k}{n}}.$$
}
\end{remark}




\end{document}